\documentclass[12pt]{amsart}
\usepackage{lmodern}
\usepackage[hidelinks]{hyperref}
\usepackage{amsmath,amssymb,amsthm}
\usepackage{enumitem}
\usepackage{graphicx}
\usepackage{color}
\usepackage{booktabs}
\usepackage{longtable}

\usepackage[a4paper, left=3cm, right=3cm, top=3cm, bottom=3cm, footskip=15mm]{geometry}

\newtheorem{theorem}{Theorem}[section]
\newtheorem{lemma}[theorem]{Lemma}
\newtheorem{proposition}[theorem]{Proposition}
\newtheorem{corollary}[theorem]{Corollary}

\theoremstyle{definition}
\newtheorem{definition}[theorem]{Definition}

\theoremstyle{remark}
\newtheorem{remark}[theorem]{Remark}

\numberwithin{equation}{section}

\newcommand{\abs}[1]{\lvert#1\rvert}

\newcommand{\inn}{~ \hat{\in}~ }

\makeindex

\usepackage{fancyhdr}
\usepackage{calc} 

\AtBeginDocument{%
  \setlength{\textwidth}{\dimexpr\paperwidth - 6cm\relax}%
  \setlength{\oddsidemargin}{\dimexpr 3cm - 1in\relax}%
  \setlength{\evensidemargin}{\dimexpr 3cm - 1in\relax}%
  \setlength{\marginparwidth}{2.5cm}%
}

\newcommand{\set}[1]{\lbrace#1\rbrace}

\newcommand{\C}{\mathcal{C}}

\newcommand{\CC}{\mathbb{C}}
\newcommand{\CV}{\mathfrak{C}}
\newcommand{\DV}{\mathfrak{D}}
\newcommand{\M}{\mathbb{M}}
\newcommand{\B}{\mathbb{B}}

\newcommand{\I}{\mathfrak{I}}
\newcommand{\E}{\mathfrak{X}}
\newcommand{\N}{\mathbb{N}}
\newcommand{\LL}{\mathbb{L}}

\newcommand{\p}{\mathbb{P}}

\newcommand{\mc}[1]{\mathcal{#1}}

\newcommand{\In}[1]{\mathrm{Int}[#1]}
\newcommand{\Ex}[1]{\mathrm{Ext}[#1]}

\begin{document}

\title{Complex circles of partition and the expansion principles}

\author{B. Gensel}
\address{Berndt Gensel: Spittal on Drau, Austria}
\email{dr.berndt@gensel.at}
\author{T. Agama}
\address{Theophilus Agama: Department of Mathematics, African Institute for mathematical sciences, Ghana.}
\email{Theophilus@aims.edu.gh/emperordagama@yahoo.com}

\subjclass[2010]{Primary 11P81; Secondary 05A17}

\date{\today}

\begin{abstract}
In this paper, we further develop the theory of circles of partition by introducing the notion of complex circles of partition. This work generalizes the classical framework, extending from subsets of positive integers as base sets to partitions defined within the complex plane, which now serves as both the base and bearing sets. We use the \emph{expansion principles} as central tools for rigorously investigating the possibility of partitioning numbers with base set as a certain subset of the complex plane.
\end{abstract}

\maketitle

\begingroup
  \setlength{\parskip}{6pt} 
  \tableofcontents
\endgroup


\section{Introduction}

The problem of additive representation lies at the intersection of number theory, combinatorics, and geometry. Classical questions such as Goldbach-type problems ask when an integer can be written as a sum of two elements from a prescribed set. They have inspired a large body of work that ranges from the Hardy--Littlewood circle method to modern additive combinatorics and the geometry of sumsets \cite{Nathanson,TaoVu,Green}. The present paper continues a geometric approach to such questions by developing a complex-plane analog of the circle-of-partition framework introduced in our earlier work \cite{CoP}.\\

The guiding idea behind the circles of partition is to encode an additive decomposition
$$
n=x+(n-x)
$$
with both summands drawn from a specified base set. In the original setting, each admissible summand is represented by a weighted point, and the complementary relation between $x$ and $n-x$ is captured by an axis that joins the corresponding points. This simple encoding has proven to be surprisingly flexible: it provides a way to track representation functions, to compare different generators, and to translate additive constraints into geometric relations among points and chords \cite{CoP}. Subsequent developments refined this perspective and led to a formulation of the squeeze principle to study the existence of additional admissible generators in Goldbach-type settings \cite{Goldbach,Squeeze}.\\

The main contribution of this paper is to extend this framework from a real number line and set-theoretic model to a complex geometric one. Instead of working only with subsets of the natural numbers, we consider the family
$$
\mathbb{C}_{\mathbb{M}}:=\{z=x+iy\in\mathbb{C}:x\in\mathbb{M},\ y\in\mathbb{R}\},
$$
where the real part is constrained by a chosen base set $\mathbb{M}\subseteq\mathbb{N}$. A complex circle of partition (cCoP) is then defined by imposing the circle condition
$$
\Im(z)^2=\Re(z)\bigl(n-\Re(z)\bigr).
$$
This condition is not merely formal: it ensures that every admissible point lies on a Euclidean circle in the complex plane with center $n/2$ on the real axis and diameter $n$. The resulting embedding circle provides a geometric envelope for the whole structure, while axis partners, conjugate points, and conjugate axes allow the real and imaginary contributions to be separated and compared in a precise way. In this sense, the complex framework retains the additive meaning of the original construction while adding a richer two-dimensional geometry.\\

A major theme of the paper is the interaction between cCoPs with different generators. We introduce interior and exterior regions relative to the embedding circle and show how these regions behave monotonically as the generator varies. This leads to a geometric comparison principle: smaller generators occupy nested interior regions, whereas larger generators push corresponding exterior regions outward. Such inclusions make it possible to compare different cCoPs without reverting to a point-by-point analysis. We also derive explicit chord-length formulas and describe how the diameter of the embedding circle plays the role of a universal axis.\\

These geometric comparisons culminate in the expansion principles. Given two nonempty cCoPs with related generators, the equality principle identifies a new admissible cCoP when the relevant axes align exactly; the forecast principle produces a larger generator; and the squeeze principle produces an intermediate one. The central theme is that information from two known cCoP can be transferred to a third cCoP whose generator lies between or beyond the original two. This is the mechanism that makes the framework useful for partition problems in restricted sets such as the primes, and is the reason the paper keeps the base set flexible throughout.\\

The present work naturally fits into the current literature on additive number theory and additive combinatorics. On the one hand, it is inspired by the classical study of the structure of sumset and inverse problems \cite{Nathanson,TaoVu}. On the other hand, it continues the program initiated in our earlier circle-of-partition papers, where additive representations were recast in geometric language and used to study Goldbach-type questions and related density problems \cite{CoP,Goldbach,Squeeze}. The purpose of this paper is to show that the same viewpoint can be lifted to the complex plane in a way that preserves the additive meaning of the original model while revealing a new geometric structure.

\subsection{Organization of the paper} The paper is organized as follows: Section~2 recalls the basic theory of circles of partition and the notation needed later. Section~3 introduces complex circles of partition, proves that their points lie on a unique embedding circle, and develops the associated chord geometry. Section~4 studies the interior and exterior points and their relation to the embedding circle. Section~5 contains the expansion principles, including the \emph{equality}, \emph{forecast}, and \emph{squeeze principles}, together with their specializations to distinguished subsets. Section~6 finally develops the limiting axes for the expansion principles.

\section{Background}

In this section, we gather the foundational background of the theory of circle of partition and related embedding method developed on the real numbers.

\begin{definition}\label{major}
Let $n\in \N$ and $\M\subseteq \N$. We denote the \emph{circle of partition} generated by $n$ with respect to the subset $\M$ by the set
\begin{align}
\C(n,\M)=\left\{[x]\mid x,n-x\in \M\right\}.\nonumber
\end{align}
We will abbreviate this as \textbf{CoP}. We call the elements of $\C(n,\M)$ points and denote them by $[x]$. In the special case where $\M = \N$, we denote CoP simply as $\C(n)$. We define $\Vert[x]\Vert := x$ as the \emph{weight} of the point $[x]$. Similarly, we denote the weight set of the points in CoP $\C(n,\M)$ by $\Vert\C(n,\M)\Vert$. Clearly, we have
\begin{align}\label{E_C(n)}
\Vert\C(n)\Vert=\lbrace 1,2,\ldots, n-1\rbrace.
\end{align}
\end{definition}
\bigskip

\begin{definition}\label{axis}
We say that $\LL_{[x],[y]}$ is an axis of the CoP $\C(n,\M)$ if and only if $x+y=n$. We say that $[y]$ is the axis partner of $[x]$, and vice versa. We do not distinguish between $\LL_{[x],[y]}$ and $\LL_{[y],[x]}$, as they represent the same axis. The point $[x]\in \C(n,\M)$ such that $2x=n$ is called the \emph{center} of CoP. If it exists, we call it a \emph{degenerate axis} $\LL_{[x]}$, in contrast to the \emph{real axes} $\LL_{[x],[y]}$. We denote the assignment of an axis $\LL_{[x],[y]}$ to a CoP $\C(n,\M)$ by
\[
\LL_{[x],[y]} \inn \C(n,\M),
\]
which implies $[x], [y] \in \C(n,\M)$ with $x+y=n$.
\end{definition}
\bigskip

From now on, we will focus solely on real axes. Thus, we will omit the term \emph{real} in the following sections.

\begin{proposition}\label{unique}
Each axis is uniquely determined by the points $[x]\in \C(n,\M)$.
\end{proposition}

\begin{proof}
Suppose that $\LL_{[x],[y]}$ is an axis of CoP $\C(n,\M)$. Suppose also that $\LL_{[x],[z]}$ is an axis with $z\neq y$. By the definition \ref{axis}, we find that $n=x+y=x+z$, which implies $y=z$. This is inconsistent with the hypothesis.
\end{proof}

\begin{proposition}
Each point of a CoP $\C(n,\M)$, except for its center, has exactly one axis partner.
\end{proposition}

\begin{proof}
Let $[x] \in \C(n,\M)$ be a point without an axis partner, assuming that $[x]$ is not the center of CoP. For every point $[y]$ and $[x]$ with $y\neq x$ and $y \in \M$, it implies
\[
x+y\neq n.
\]
This is inconsistent with the definition \ref{major}, since $[x] \in \C(n,\M)$. By Proposition \ref{unique}, the possibility of more than one axis partner is impossible.
\end{proof}

\paragraph*{\textbf{Notation}}
We denote the number of real axes of CoP $\C(n,\M)$ by 
\begin{align}\label{axesOfCoP}
\nu(n,\M):=\abs{\{\LL_{[x],[y]}\inn \C(n,\M)\}}.
\end{align}
It is evident that
\[
\nu(n,\M)=\left\lfloor\frac{k}{2}\right\rfloor
\]
if CoP $\C(n,\M)$ has $k$ members.
\bigskip

It is not clear whether the axes $\LL_{[x],[y]}\inn \C(n,\M)$ are lines that join points $[x]$ and $[y]$ on a CoP. In the present study, we will see that a transition from the base set $\M\subseteq \N$ into a certain subset $\CC_{\M}$ of the complex plane reveals this natural geometric feature of circles of partition.

\section{Complex Circles of Partition}

Here, we introduce a special subset of complex numbers to be used as the base set of CoPs.

\begin{definition}\label{D_cCoP}
Let $\M \subseteq \N$ and
\[
\CC_{\M}:=\lbrace z=x+iy\mid x\in \M, \, y\in \mathbb{R}\rbrace \subset \CC
\]
be a subset of the complex numbers where the real part is from $\M \subseteq \N$. A CoP with special requirement
\[
\C^o(n,\CC_{\M})=\lbrace [z]\mid z, n-z\in \CC_{\M}, \,\Im(z)^2 = \Re(z)\left(n-\Re(z)\right)\rbrace
\]
will be referred to as a \emph{complex circles of partition}, abbreviated as \textbf{cCoP}. The special requirement is called the \emph{circle condition}. The components $x$ and $y$ are referred to as the \emph{real weight} and \emph{imaginary weight}, respectively. The CoP $\C(n,\M)$ is called the \emph{source CoP}. Since in the case $\M = \N$ the source CoP is abbreviated as $\C(n)$, we will set
\begin{align}
\C^o(n):=\C^o(n, \CC_{\N}).\label{E-cCoP(N)}
\end{align}
To distinguish between the points $[z]$ of cCoPs and the points $z$ in the complex plane $\CC$, we refer to the latter as \emph{complex points}.
\end{definition}

\begin{definition}
Let $\C^o(n,\CC_{\M})$ be a cCoP and $[z]\in \C^o(n, \CC_{\M})$ be such that $z=x+iy$. The point $[n-z]$ with the weight $(n-x)-iy$ denotes the axis partner of $[z]$.
\end{definition}

In this framework, the first requirement of a CoP is fulfilled:
\[
\|[z]\|+\|[n-z]\|=x+iy+n-x-iy=n.
\]

\emph{Important:} For axis partners $[z_1]$ and $[z_2]=[n-z_1]$, we always have
\begin{align}\label{E_Im-axispartner}
\Im(z_1)=-\Im(z_2).
\end{align}

\begin{definition}
Let $\C^o(n,\CC_{\M})$ be a cCoP and $[z]\in \C^o(n, \CC_{\M})$ be such that $z=x+iy$. The point $[\overline{z}]$ with the weight $x-iy$ denotes the \emph{conjugate point} of $[z]$.
\end{definition}

\begin{definition}\label{D_conjugateAxis}
Let $\C^o(n,\CC_{\M})$ be a cCoP and $\LL_{[z],[n-z]}\inn \C^o(n, \CC_{\M})$. We denote the \emph{conjugate axis} of $\LL_{[z],[n-z]}$ by 
\[
\LL_{[\overline{z}],[\overline{n-z}]}.
\]
We do not distinguish between axes $\LL_{[z],[n-z]}$ and $\LL_{[n-z],[z]}$, since we do not consider the axes to be different up to the rearrangement of resident points $[z]$ and $[n-z]$.
\end{definition}

\begin{definition}
In relation to the definition \ref{axesOfCoP}, we define
\[
\nu^o(n,\CC_{\M}):=\abs{\{\LL_{[z],[n-z]}\inn \C^o(n, \CC_{\M})\}}
\]
as the number of axes of the cCoP $\C^o(n,\CC_{\M})$. Clearly,
\begin{align}
\nu^o(n,\CC_{\M})&= 
\begin{cases}
2\nu(n,\M)\quad \text{if the CoP} \; \C(n, \M) \text{ does not contain a degenerate axis}, \\
2\nu(n,\M) + 1 \quad \text{if the CoP} \; \C(n, \M) \text{ contains a degenerate axis}.
\end{cases}
\end{align}
\end{definition}

\noindent
In the following, we will deduce that the \emph{circle condition}
\begin{align}\label{E_spRequirement}
\Im(z)^2=\Re(z)\left(n-\Re(z)\right)
\end{align}
ensures that all points of a cCoP lie on a circle in the complex plane $\CC$.

\begin{theorem}\label{T_cCoPonCircle}
Let $\C^o(n, \CC_{\M})$ be a non-empty cCoP. The weights of all its points are located on a circle in the complex plane $\CC$ with its center on the real axis at $\frac{n}{2}$ and with a diameter of length $n$.
\end{theorem}

\begin{proof}
We consider an arbitrary point $[z] \in \C^o(n, \CC_{\M})$ and its axis partner $[n-z]$. We set $x:=\Re(z)$ and $y:=\Im(z)$
\footnote{This framework will be used in the sequel.}. Using the circle condition \eqref{E_spRequirement}, we have
\begin{align}\label{E_imaginary}
y^2=x(n-x).
\end{align}
The definition \ref{D_cCoP} implies $x\in \M \subseteq \N$. Hence, $x>0$. We obtain for the second requirement for $[z]\in \C^o(n,\CC_{\M})$ the property $n-x\in \M$. Therefore, the weight $x$ must satisfy $0<x<n$. We now find the greatest imaginary part of the complex point $z_0$ such that $[z_0]\in \C^o(n,\CC_{\M})$. This means finding the root of the derivative of \eqref{E_imaginary}:
\begin{align*}
\frac{dy}{dx}&=\frac{d}{dx}\sqrt{x(n-x)} \\
&=\frac{1}{2}\frac{n-2x}{\sqrt{x(n-x)}}=0.
\end{align*}
Thus, we obtain $x_0=\frac{n}{2}$, provided that the denominator is non-vanishing. In this deduction, the denominator cannot be zero. Substituting $x=\frac{n}{2}$ into \eqref{E_imaginary}, we get
\[
y_0^2=\frac{n}{2}\left(n-\frac{n}{2}\right)=\left(\frac{n}{2}\right)^2
\]
and, hence $\abs{y_0}=\abs{\Im(z_0)}=\frac{n}{2}$. Clearly, $\Im(n-z)=\Im(\overline{z})=-\Im(z)$. Therefore, the points $[z]$, $\overline{[z]}$, and $[n-z]$ form a right-angled triangle with the hypotenuse $\LL_{[z], [n-z]}$ and the legs $2y$ and $n-2x$. By the Pythagorean Theorem (see Figure \ref{fig:diameter-axis-ccop}), we have
\begin{align*}
\abs{\LL_{[z], [n-z]}}^2 &=(2y)^2+(n-2x)^2 \\
&\text{and using \eqref{E_imaginary}} \\
&=4nx-4x^2+n^2-4nx+4x^2 \\
&=n^2\text{ and thus} \\
\abs{\LL_{[z],[n-z]}}&= n.
\end{align*}
Since the sum of $z$ and $n-z$ equals $n$, both points $[z]$ and $[n-z]$ are the end points of an axis $\LL_{[z],[n-z]}\inn \C^o(n, \CC_{\M})$ and simultaneously form the diameter of a circle containing the complex points $z$, $\overline{z}$, and $n-z$, as their imaginary parts satisfy the circle condition. This is a circle with a center on the real axis at $\frac{n}{2}$ and a diameter of length $n$.
\end{proof}
\bigskip

\begin{figure}[ht]
\centering
   \includegraphics[width=240pt]{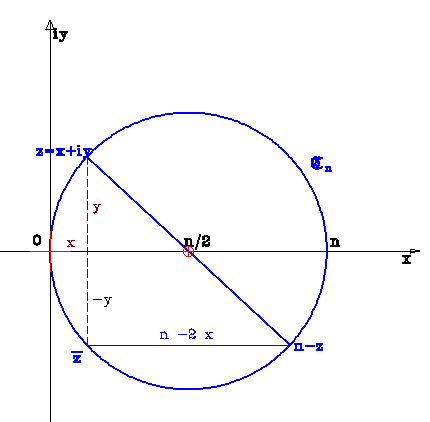}
\caption{Diameter as axis of a cCoP.}
\label{fig:diameter-axis-ccop}
\end{figure}

\begin{definition}\label{D_embeddingCircle}
The circle in the complex plane $\CC$ with center on the real axis at $\frac{n}{2}$ and diameter of length $n$ will be called the \emph{embedding circle} and denoted by $\CV_n$ of the cCoP $\C^o(n,\CC_{\M})$. It implies that
\begin{align*}
\CV_n&=\set{z\in\CC \mid 0\leq \Re(z)\leq n,\Im(z)^2=\Re(z)(n-\Re(z))}.
\end{align*}
Furthermore, we define
\begin{align*}
\I_n&:=\set{z\in \CC \mid 0\leq \Re(z)\leq n,~\Im(z)^2<\Re(z)(n-\Re(z))}, \\
\E_n &:=\CC \setminus (\I_n\cup \CV_n)
\end{align*}
as the sets of all complex points $z\in \CC$ within and outside of the embedding circle $\CV_n$, respectively.
\end{definition}

\begin{definition}\label{D-diameter}
The diameter of the embedding circle $\CV_n$ passing through the complex points $z$ and $n-z$ is denoted by
\[
\DV_n(z,n-z).
\]
\end{definition}

It is obvious that for a nonempty cCoP $\C^o(m,\CC_{\M})$ with $m<n$, the following holds:
\begin{align}\label{E_cCoPembedding}
||\C^o(m,\CC_{\M})||&\subset \CV_m \subset \I_n, \\
\I_m &\subset \I_n \quad \text{and} \quad \E_n \subset \E_m. \nonumber
\end{align}

\begin{corollary}
For all subsets $\M \subseteq \N$, the cCoPs $\C^o(n,\CC_{\M})$ for a fixed generator $n$ have a unique embedding circle $\CV_n$.
\end{corollary}

\begin{proposition}\label{L_emptyCap}
Let $\CV_m$ and $\CV_n$ be two embedding circles with $m\neq n$. The circles $\CV_m$ and $\CV_n$ have the origin as their only common point. In particular, we have
\[
\CV_m \cap \CV_n = \{(0,0)\}.
\]
\end{proposition}

\begin{proof}
Suppose that $z_m \in \CV_m$ and that $z_n \in \CV_n$. Assume also that $z_m=z_n$ is a common complex point of both circles. We deduce $\Re(z_m)=\Re(z_n)$. By the circle condition \eqref{E_spRequirement}, we get for the imaginary parts
\begin{align*}
\Im(z_m)^2 &=\Re(z_m)\left(m-\Re(z_m)\right), \\
\Im(z_n)^2 &=\Re(z_n)\left(n-\Re(z_n)\right)=\Re(z_m)\left(n-\Re(z_m)\right),
\end{align*}
and as a difference:
\[
\Im(z_m)^2-\Im(z_n)^2=\Re(z_m)(m-n)=\Re(z_n)(m-n).
\]
Since $m\neq n$, this is only zero if $\Re(z_m)=\Re(z_n)=0$. Thus, for the imaginary part, we get (by the circle condition)
\[
\Im(z_m)^2=0(m-0)=0(n-0)=\Im(z_n)^2.
\]
Hence, the origin is the only common point of $\CV_m$ and $\CV_n$.
\end{proof}

\begin{corollary}[Big Bang]\label{C_bigBang}
If $m<n$, then the circle $\CV_m$ resides fully inside the circle $\CV_n$, except at the common origin. Therefore, the origin is the only common complex point of \textbf{all} embedding circles with increasing diameters, serving as the "Big Bang" of all embedding circles (refer to Figure \ref{fig:big-bang}).
\end{corollary}
\bigskip


\begin{figure}[ht]
\centering
   \includegraphics[width=240pt]{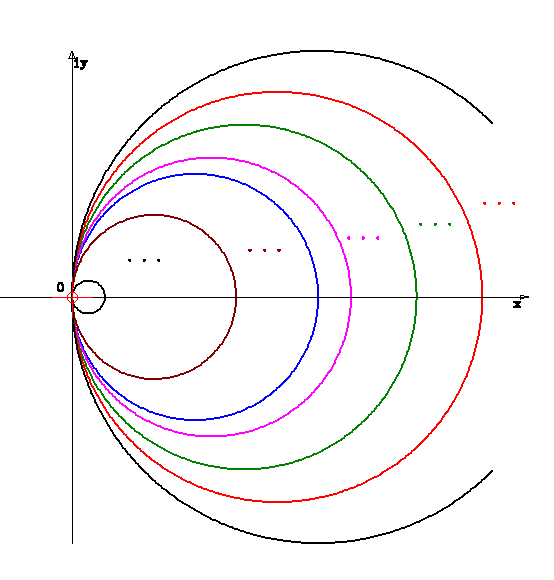}
\caption{Figure 2. The ``Big Bang''.}
\label{fig:big-bang}
\end{figure}

\begin{theorem}
Let $\C^o(m,\CC_{\M})$ and $\C^o(n,\CC_{\M})$ be two non-empty cCoPs with $m\neq n$. The cCoPs $\C^o(m,\CC_{\M})$ and $\C^o(n,\CC_{\M})$ have no common point
\[
\C^o(m,\CC_{\M})\cap\C^o(n,\CC_{\M})=\emptyset.
\]
\end{theorem}

\begin{proof}
Due to \eqref{E_cCoPembedding} and by Proposition \ref{L_emptyCap} the origin could be the only common point of both cCoPs. Since $\M\subseteq\N$, the real weight of a point of any cCoP cannot be $0$. Hence, both cCoPs $\C^o(m,\CC_{\M})$ and $\C^o(n,\CC_{\M})$ do not have a common point.  
\end{proof}

\begin{proposition}\label{P_equMembers}
Let $\C^o(m,\CC_{\M})$ and $\C^o(n,\CC_{\M})$ be two nonempty cCoPs with $n\neq m$. There are points $[z_m]\in\C^o(m,\CC_{\M})$ and $[z_n]\in\C^o(n,\CC_{\M})$ with a common real weight $\Re(z_m)=\Re(z_n)=x\in\M$ if and only if their source CoPs $\C(m,\M)$ and $\C(n,\M)$ share a common point $[x]$.
\end{proposition}

\begin{proof}
Let $[x]$ be a common point of $\C(m,\M)$ and $\C(n,\M)$. We find that $m-x$ and $n-x$ are members of $\M$, and $m-x-iy_m$ and $n-x-iy_n$ are members of $\CC_{\M}$. Consequently, their axis partners $x+iy_m$ and $x+iy_n$ are also members of $\CC_{\M}$. Therefore, with $z_m=x+iy_m$ and $z_n=x+iy_n$, we have
\[
[z_m]\in\C^o(m,\CC_{\M})\mbox{ and }[z_n]\in\C^o(n,\CC_{\M})
\]
with $x=\Re(z_m)=\Re(z_n)$. This reasoning can be reversed so that from $x=\Re(z_m)=\Re(z_n)$, we deduce $[x]\in\C(m,\M)\cap\C(n,\M)$. 
\end{proof}
\bigskip

\begin{corollary}\label{C-c_equ_source}
The cCoP $\C^o(n,\CC_{\M})$ is not empty if and only if its source CoP $\C(n,\M)$ is not empty.
\end{corollary}
\bigskip

\begin{proposition}\label{P-cCoP(N)}
Let $\M:=\N$ be the set of all positive integers. Each cCoP $\C^o(n)$ for integers $n\in \mathbb{N}$ with $n\geq 2$ is not empty.
\end{proposition}

\begin{proof}
By definition \ref{major}, the source CoPs of such cCoPs are $\C(n)$. These are non-empty for all integers $n\geq 2$ by virtue of \eqref{E_C(n)}. Due to Corollary \ref{C-c_equ_source}, their corresponding cCoPs are also nonempty.
\end{proof}

\begin{corollary}
Let $[z]$ be an arbitrary point (as in Theorem \ref{T_cCoPonCircle}) in a cCoP $\C^o(n,\CC_{\M})$. Each axis of the form $\LL_{[z],[n-z]}$ of a cCoP $\C^o(n,\CC_{\M})$ has equal lengths. In particular,
\[
\abs{\LL_{[z],[n-z]}}=n\mbox{ for all points }[z]\in\C^o(n,\CC_{\M}).
\]
\end{corollary}
\bigskip

We now deduce the length of a chord that joins arbitrary points of a cCoP under the circle condition. This may be viewed as a generalization of the length of an axis of a cCoP.

\begin{theorem}
Let $\C^o(n,\CC_{\M})$ be a nonempty cCoP, and let $[z_1], [z_2]\in\C^o(n,\CC_{\M})$, where $z_1:=x_1+iy_1$ and $z_2:=x_2+iy_2$. Denote the length of the chord $\mc{L}_{[z_1],[z_2]}$ by $\Gamma([z_1],[z_2])$. We have
\footnote{See \cite[p.2]{CoP} definition 2.2}
\[
\abs{\mc{L}_{[z_1],[z_2]}}=
\Gamma([z_1],[z_2])=
\abs{\sqrt{x_1(n-x_2)}\pm \sqrt{x_2(n-x_1)}},
\]
where the ''$-$'' sign is used if $\mathrm{sign}(y_1)=\mathrm{sign}(y_2)$ and the ''$+$'' sign otherwise.
\end{theorem}

\begin{proof}
The deduction of the length of the chord that joins two arbitrary points on a cCoP is straightforward. In particular, we deduce
\begin{align*}
\abs{\mc{L}_{[z_1],[z_2]}}^2
&=(x_1-x_2)^2+(y_1-y_2)^2\\
&=x_1^2+x_2^2-2x_1x_2+y_1^2+y_2^2\pm 2\abs{y_1y_2}\\
&\mbox{and using \eqref{E_imaginary}}\\
&=x_1^2+x_2^2-2x_1x_2\pm 2\abs{y_1y_2}+nx_1-x_1^2+nx_2-x_2^2\\
&=nx_1-x_1x_2+nx_2-x_1x_2\pm 2\abs{y_1y_2}\\
&=x_1(n-x_2)+x_2(n-x_1)\pm 2\sqrt{x_1(n-x_1)}\cdot\sqrt{x_2(n-x_2)}\\
&=x_1(n-x_2)+x_2(n-x_1)\pm 2\sqrt{x_2(n-x_1)}\cdot\sqrt{x_1(n-x_2)}\\
&=\left(\sqrt{x_1(n-x_2)}\pm \sqrt{x_2(n-x_1)} \right)^2.
\end{align*}
Thus, the function $\Gamma([z_1],[z_2])$ for the length of the chord simplifies to:
\begin{align}\label{E_Gamma}
\Gamma([z_1],[z_2])=\abs{\sqrt{x_1(n-x_2)}\pm \sqrt{x_2(n-x_1)}}.
\end{align}
\end{proof}
\bigskip

The following deductions can be made from the preceding calculation: If $[z_2]$ and $[n-z_1]$ are points in the cCoP that satisfy $z_2=n-z_1$, then the chord $\mc{L}_{[z_1],[z_2]}$ is a diameter. In this case, $y_2=-y_1$ and $x_2=n-x_1$, and therefore
\begin{align*}
\Gamma([z_1],[n-z_1])
&=\abs{\sqrt{x_1(n-x_2)}+\sqrt{x_2(n-x_1)}}\\
&=\abs{\sqrt{x_1 x_1}+\sqrt{(n-x_1)(n-x_1)}}\\
&=\abs{x_1+n-x_1}=n.
\end{align*}

If $[z_2]$ is the axis partner of the conjugate point of $[z_1]$, then $x_2=n-x_1$ and $y_2=y_1$. Since the signs of both $y$ values are equal, we obtain
\begin{align*}
\Gamma([z_1],[n-z_1])
&=\abs{\sqrt{x_1 x_1}-\sqrt{x_2 x_2}}\\
&=\abs{x_1 - x_2}.
\end{align*}
This observation coincides with the length of a chord in a CoP formulated in \cite{CoP}.\\

The degenerate axis of a cCoP coincides with the diameter that is parallel to the imaginary axis. It is a real diameter but has the property that it is equal to the conjugate axis. In this case, using \eqref{E_Gamma} with $x_2=x_1=\frac{n}{2}$ and $y_2=-y_1$, we have
\begin{align*}
\Gamma([z_1],[n-z_1])
&=\left|\sqrt{\left(\frac{n}{2}\right)^2}+\sqrt{\left(\frac{n}{2}\right)^2}\right|=n.
\end{align*}
\bigskip

\section{The interior and exterior points of complex circles of partition}

In this section, we introduce and study the concept of \emph{interior} and \emph{exterior} points of the complex circles of partition.

\begin{theorem}\label{T_interior}
Let $\C^o(n,\CC_{\M})$ be a nonempty cCoP. The distance from every complex point of $||\C^o(n,\CC_{\M})||$ to every complex point in $\I_n$ is less than $n$, and there are complex points in $||\C^o(n,\CC_{\M})||$ whose distances to some complex point in $\E_n$ are greater than $n$.
\end{theorem}

\begin{proof}
The diameter of $\CV_n$ is the longest line from any complex point on this circle to any complex point inside or on the circle. Hence, all complex points of $\I_n$ have a smaller distance to any complex point on $\CV_n$ than the diameter length. By \eqref{E_cCoPembedding}, this relation is also valid for any complex points of $||\C^o(n,\CC_{\M})||$ and $\I_n$. Therefore, their distances are less than the length of the diameter of $\CV_n$, which is $n$. On the other hand, the distances between some complex point of $||\C^o(n,\CC_{\M})||$ and some complex point in $\E_n$ are greater than $n$, since $\E_n$ consists of complex points outside of the embedding circle $\CV_n$ and the members of $\E_n$ are unbounded. 
\end{proof}
\bigskip

\begin{corollary}
For two non-empty cCoPs $\C^o(m,\CC_{\M})$ and $\C^o(n,\CC_{\M})$ with $m<n$, all distances between points of these sets are less than $n$. Furthermore, there are some points in these cCoPs whose mutual distances are greater than $m$.
\end{corollary}
\bigskip

\begin{definition}\label{D_IntExt}
Let $\I_n,\E_n$ be (as in definition \ref{D_embeddingCircle}) all complex points \emph{inside} and \emph{outside} the embedding circle $\CV_n$. We call the points $z\in\I_n\cap\CC_{\M}$ \emph{interior} points with respect to $\CV_n$ and denote the set of all such points by $\In{\CV_n}$. Similarly, we call the complex points $z\in\E_n\cap\CC_{\M}$ the \emph{exterior} points with respect to $\CV_n$ and denote the set of all these points by $\Ex{\CV_n}$.
\bigskip

We can observe that
\begin{align*}
\In{\CV_n}=\I_n\cap\CC_{\M}\quad\text{and}\quad
\Ex{\CV_n}=\E_n\cap\CC_{\M}.
\end{align*}
\end{definition}

\begin{definition}
Let $\C^o(n,\CC_{\M})$ be a non-empty cCoP, and let $\CV_n$ be its embedding circle. We call the complex point $z\in\In{\CV_n}$ an \emph{interior} point with respect to cCoP $\C^o(n,\CC_{\M})$ if and only if for \emph{all} points $[w]\in\C^o(n,\CC_{\M})$ implies $\abs{z-w}<n$. We denote the set of all such points by $\In{\C^o(n,\CC_{\M})}$. Similarly, we call the complex point $z\in\Ex{\CV_n}$ an \emph{exterior} point with respect to $\C^o(n,\CC_{\M})$ if for \emph{some} points $[w]\in\C^o(n,\CC_{\M})$ implies $\abs{z-w}>n$. We denote the set of all such points by $\Ex{\C^o(n,\CC_{\M})}$.
\end{definition}
\bigskip

\begin{remark}
Let $n_0\in\N$ be the least generator for all cCoPs.
If $n>n_0$ and $\C^o(n,\CC_{\M})$ is an empty cCoP, then $\In{\C^o(n,\CC_{\M})}$ and $\Ex{\C^o(n,\CC_{\M})}$ are empty by definition.
\end{remark}

\begin{corollary}
If $\C^o(n,\CC_{\M})$ is a nonempty cCoP, then (by Theorem \ref{T_interior})
\begin{align}\label{E_IntExt}
\In{\C^o(n,\CC_{\M})}&=\In{\CV_n}=\I_n\cap\CC_{\M}\nonumber\\
&\mbox{ and }\\
\Ex{\C^o(n,\CC_{\M})}&=\Ex{\CV_n}=\E_n\cap\CC_{\M}\nonumber.
\end{align}
\end{corollary}
\bigskip

\begin{proposition}\label{P_IntExt}
Let $\C^o(m,\CC_{\M})$ and $\C^o(n,\CC_{\M})$ be two nonempty cCoPs. We have $m<n$ if and only if
\[
\In{\C^o(m,\CC_{\M})}\subset\In{\C^o(n,\CC_{\M})}\mbox{ and }
\Ex{\C^o(n,\CC_{\M})}\subset\Ex{\C^o(m,\CC_{\M})}.
\]
\end{proposition}

\begin{proof}
We suppose $m<n$. By \eqref{E_IntExt}, we have
\begin{align*}
\In{\C^o(m,\CC_{\M})}&=\I_m\cap\CC_{\M}\mbox{ and since \eqref{E_cCoPembedding}}\\
&\subset\I_n\cap\CC_{\M}
=\In{\C^o(n,\CC_{\M})}.
\end{align*}
Similarly, we have
\begin{align*}
\Ex{\C^o(n,\CC_{\M})}&=\E_n\cap\CC_{\M}\mbox{ and since \eqref{E_cCoPembedding}}\\
&\subset\E_m\cap\CC_{\M}
=\Ex{\C^o(m,\CC_{\M})}.
\end{align*}
On the other hand, with the embedding $\In{\C^o(m,\CC_{\M})}\subset\In{\C^o(n,\CC_{\M})}$, we deduce $\I_m\cap\CC_{\M}\subset\I_n\cap\CC_{\M}$, which is only valid with $m<n$. Analogously, using the embedding $\Ex{\C^o(n,\CC_{\M})}\subset\Ex{\C^o(m,\CC_{\M})}$, we deduce $m<n$.
\end{proof}
\bigskip

\begin{proposition}
Let $\C^o(m,\CC_{\M})$ and $\C^o(n,\CC_{\M})$ be two nonempty cCoPs. We have $m<n$ if and only if
\[
||\C^o(m,\CC_{\M})||\subset\In{\C^o(n,\CC_{\M})}\mbox{ and }
||\C^o(n,\CC_{\M})||\subset\Ex{\C^o(m,\CC_{\M})}.
\]
\end{proposition}

\begin{proof}
We suppose $m<n$. By \eqref{E_cCoPembedding} and the embedding $||\C^o(m,\CC_{\M})||\subset\CC_{\M}$, we deduce
\begin{align*}
||\C^o(m,\CC_{\M})||&\subset\CV_m\cap\CC_{\M}\\
&\subset(\CV_m\cap\CC_{\M})\cup\I_m\\
&\subset(\CV_m\cup\I_n)\cap\CC_{\M}\mbox{ and since }\CV_m\subset\I_n\\
&=\I_n\cap\CC_{\M}\mbox{ and because of \eqref{E_IntExt}}\\
&=\In{\C^o(n,\CC_{\M})}.
\end{align*}
The embedding $||\C^o(n,\CC_{\M})||\subset\Ex{\C^o(m,\CC_{\M})}$ can be easily verified. On the other hand, embedding $||\C^o(m,\CC_{\M})||\subset\In{\C^o(n,\CC_{\M})}$ implies $\I_m\cap\CC_{\M}\subset\I_n\cap\CC_{\M}$, which is only valid for $m<n$. Analogously, using the embedding $\Ex{\C^o(n,\CC_{\M})}\subset\Ex{\C^o(m,\CC_{\M})}$, we deduce $m<n$.
\end{proof}
\bigskip

\begin{proposition}
Let $\C^{o}(m,\CC_{\M})\neq \emptyset$. If $[z_1],[z_2]$ are the axis partners of cCoP $\C^{o}(n,\CC_{\M})$ and $|\LL_{[z_1],[z_2]}|=n>m$, then $z_1,z_2\in \mathrm{Ext}[\C^{o}(m,\CC_{\M})]$. 
\end{proposition}

\begin{proof}
Using the requirement $\LL_{[z_1],[z_2]}\inn\C^o(n,\CC_{\M})$ with $n>m$ and Proposition \ref{P_IntExt}, we deduce
\begin{align*}
||\C^o(n,\CC_{\M})||&\subset\Ex{\C^o(m,\CC_{\M})}
\mbox{ and therefore}\\
z_1,z_2&\in\Ex{\C^o(m,\CC_{\M})}.
\end{align*}
\end{proof}

\begin{proposition}
Let $\C^{o}(m,\CC_{\M})\neq \emptyset$. If $\mathrm{Int}[\C^{o}(m,\CC_{\M})]\subset \mathrm{Int}[\C^{o}(n,\CC_{\M})]$, then $\C^{o}(n,\CC_{\M})\neq \emptyset$.
\end{proposition}

\begin{proof}
The above hypothesis with definition \ref{D_IntExt} implies that $\In{\C^o(m,\CC_{\M})}\neq\emptyset$ and $\In{\C^o(n,\CC_{\M})}\supset\emptyset$. Hence $\C^o(n,\CC_{\M})\neq\emptyset$.
\end{proof}

We state a converse of the above result in the following theorem.

\begin{theorem}
Let $\C^{o}(m,\CC_{\M}),\C^{o}(n,\CC_{\M})\neq \emptyset$. If $m<n$ and $[z]\inn \C^{o}(n,\CC_{\M})$, then  $z\not \in \In{\C^o(m,\CC_{\M})}$.
\end{theorem}

\begin{proof}
By definition \ref{D_embeddingCircle}, it implies $\CV_n\cap\I_n=\emptyset$ and $||\C^o(n,\CC_{\M})||\subset\CV_n$. It easily follows that $\I_n\cap||\C^o(n,\CC_{\M})||=\emptyset$. We know that for each point $[z]\in\C^o(n,\CC_{\M})$ the point  $z\not\in\I_n$. Because $m<n$, we deduce additionally that $\I_m\subset\I_n$ and hence 
\[
z\not\in\I_n\supset\I_m\supset\I_m\cap\CC_{\M}=\In{\C^o(m,\CC_{\M})}.
\]
\end{proof}

\section{The expansion principles}

In this section, we do not distinguish between the axes $\LL_{[z],[n-z]}$ and $\LL_{[n-z],[z]}$, as the axes are considered equivalent under the rearrangement of resident points. Subsequently, we will consider only axes $\LL_{[z],[n-z]}$ with
\[
\Re(z)<\Re(n-z).
\]

\begin{lemma}[The axial point ordering principle]\label{L-AxialPOP}
Let $\M \subseteq \N$ and $\C^o(n,\CC_{\M})$ and $\C^o(n+t,\CC_{\M})$ be non--empty cCoPs with integers $ n,t\in\mathbb{Q}\subseteq\N $. Consider the axes $ \LL_{[z],[n-z]} \inn \C^o(n,\CC_{\M})$ and $ \LL_{[w],[n+t-w]} \inn \C^o(n+t,\CC_{\M}) $. We have
\begin{align}\label{E_equiv1}
\Re(z)<\Re(w) \mbox{ and } \Re(n-z)<\Re(n+t-w)
\end{align}
if and only if
\begin{align}\label{E_equiv2}
\Re(z)<\Re(w)<\Re(z)+t.
\end{align}
\end{lemma}

\begin{proof}
We note that the left inequalities are equivalent. Hence, we need to demonstrate that the right inequalities are also equivalent. Initially, we assume \eqref{E_equiv1}. From the right inequality and the existence of $ \LL_{[w],[n+t-w]} \inn \C^o(n+t,\CC_{\M}) $, we obtain
\begin{align*}
\Re(n-z)&<\Re(n+t-w)=n+t-\Re(w)\\
\Longrightarrow
\Re(w)&<n+t-\Re(n-z)=\Re(z)+t.
\end{align*}
This corresponds to the right side of \eqref{E_equiv2}. Conversely, if the right side of \eqref{E_equiv2} holds, we combine with $\Re(w)=n+t-\Re(n+t-w)$ and get
\begin{align*}
\Re(w)=n+t-\Re(n+t-w)&<\Re(z)+t=n-\Re(n-z)+t\\
\Longrightarrow \Re(n-z)&<\Re(n+t-w).
\end{align*}
This establishes the right inequality of \eqref{E_equiv1}.
\end{proof}
\bigskip


\begin{figure}[htbp]
\centering
   \includegraphics[width=240pt]{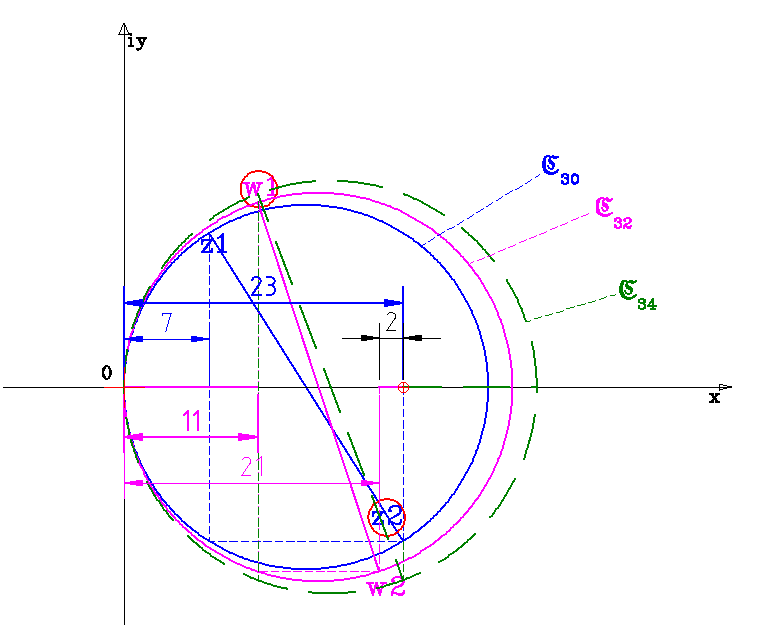}
   \caption{Forecasting of $\C^o(34,\CC_{\p})$ by $\C^o(30)$ and $\C^o(32)$} \label{F-forecasting}
\end{figure}

Here, we state some fundamental results concerning the existence of new nonempty cCoPs that depend on two known nonempty cCoPs.

\begin{theorem}[The expansion principles]\label{T-Expansion}
Let $\B\subset\M\subseteq\N$ and $n,t,s\in\mathbb{Q}\subseteq\N$. Furthermore, let $\C^o(n,\CC_{\M})$ and $\C^o(n+t,\CC_{\M})$
be two nonempty cCoPs with the axes 
\[
\LL_{[z],[n-z]}\inn\C^o(n,\CC_{\M})\mbox{ and }\LL_{[w],[n+t-w]} \inn \C^o(n+t,\CC_{\M}). 
\]
If 
\[
w,n-z\in\CC_{\B}\mbox{ and }\Re(w)=\Re(z)+s, 
\]
then the axis $\LL_{[u],[n+s-u]}\inn \C^o(n+s,\CC_{\B})$ exists with
\[
\Re(u)=\Re(w) \mbox{ and }\Re(n+s-u)=\Re(n-z)
\]
and the imaginary parts corresponding to the \emph{circle condition}. Thus, $\C^o(n+s,\CC_{\B})$ is a nonempty cCoP.
\end{theorem}

\begin{proof}
Using the requirement $\Re(w)=\Re(z)+s$, we get
\begin{align}
n+t-\Re(n+t-w)&=n-\Re(n-z)+s\nonumber\\
\Re(n+t-w)&=\Re(n-z)+t-s.\label{E-Re(n-z)}
\end{align}
Since $w\in\CC_{\B}$ and $n-z\in\CC_{\B}$, we have (by \eqref{E-Re(n-z)})
\[
\Re(w)+\Re(n-z)=\Re(w)+\Re(n+t-w)-t+s=n+s.
\]
Hence, there exists an axis $\LL_{[u],[n+s-u]}\inn\C^o(n+s,\CC_{\B})$ with
\begin{align*}
\Re(u)&=\Re(w) \mbox{ and }\\
\Re(n+s-u)&=n+s-\Re(u)\\
&=n+s-\Re(w)\\
&=n+s-(\Re(z)+s)\\
&=n-\Re(z)=\Re(n-z)
\end{align*}
and the imaginary parts corresponding to the \textit{circle condition}. Thus, cCoP $\C^o(n+s,\CC_{\B})$ is not empty.
\end{proof}
\bigskip

Depending on the range of values of $s$, we can deduce the following:

\subsection{The equality principle}

If $s=t$, then the axis 
\[
\LL_{[w],[n+t-w]}\equiv\LL_{[u],[n+t-u]}\inn\C^o(n+t,\CC_{\M})
\] 
is an axis of $\C^o(n+t,\CC_{\B})$. Furthermore, by \eqref{E-Re(n-z)}, we deduce
\[
\Re(z)<\Re(w)\mbox{ and }\Re(n-z)=\Re(n+t-w).
\]

\subsection{The forecast principle}

If $s>t$, then we have determined from two known nonempty cCoP $\C^o(n,\CC_{\M})$ and $\C^o(n+t,\CC_{\M})$ a new nonempty cCoP $\C^o(n+s,\CC_{\B})$ with a generator $>n,n+t$. By \eqref{E-Re(n-z)}, we deduce
\[
\Re(z)<\Re(w)\mbox{ and }\Re(n-z)>\Re(n+t-w).
\]

\subsection{The squeeze principle}

If $0<s<t$, then we have determined from two known nonempty cCoP $\C^o(n,\CC_{\M})$ and $\C^o(n+t,\CC_{\M})$ a new nonempty cCoP $\C^o(n+s,\CC_{\B})$ with a generator $n+s$ satisfying $n<n+s<n+t$. By \eqref{E-Re(n-z)}, we deduce
\[
\Re(z)<\Re(w)\mbox{ and }\Re(n-z)<\Re(n+t-w)
\]
and by Lemma \ref{L-AxialPOP}
\[
\Re(z)<\Re(w)<\Re(z)+t.
\]
\bigskip


\begin{figure}[htbp]
\centering
   \includegraphics[width=240pt]{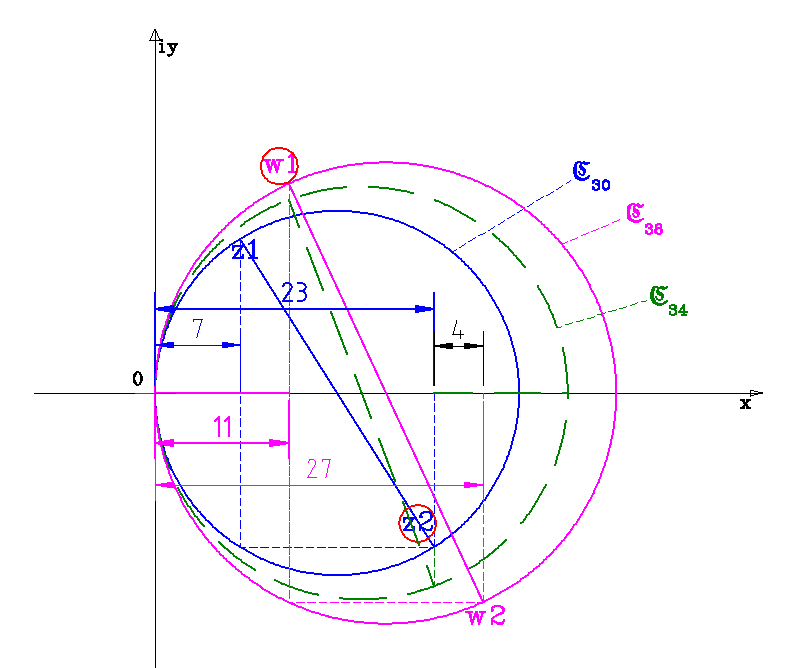}
   \caption{Squeezing of $\C^o(34,\CC_{\p})$ by $\C^o(30)$ and $\C^o(38)$ }
\label{F-squeeze}
\end{figure}

If we replace the general subset $\M$ by its parent set $\N$ and the subset $\B$ by the set of prime numbers $\p\subset\N$, we get the following consequences:

\begin{corollary}[The special expansion principles]
Let $n,t,s\in\mathbb{Q}\subseteq\N$, and let $\C^o(n)$ and $\C^o(n+t)$ be two nonempty cCoPs with the axes
\[
\LL_{[z],[n-z]}\inn\C^o(n)\mbox{ and } 
\LL_{[w],[n+t-w]} \inn \C^o(n+t) 
\]
with the base set $\CC_{\N}$\footnote{see \eqref{E-cCoP(N)}}.
If
\[
w,n-z\in\CC_{\p}\mbox{ and }\Re(w)=\Re(z)+s, 
\]
then the axis $\LL_{[u],[n+s-u]}\inn \C^o(n+s,\CC_{\p})$ exists with
\[
\Re(u)=\Re(w) \mbox{ and }\Re(n+s-u)=\Re(n-z)
\]
and the imaginary parts corresponding to the \textit{circle condition}. Thus, $\C^o(n+s,\CC_{\p})$ is a nonempty cCoP.
\end{corollary}
\bigskip

\section{Limiting axes} 

In contrast to the previous section, we will treat an axis of a cCoP $\C^o(n,\CC_{\M})$ as $\LL_{[z_1],[z_2]}$ instead of $\LL_{[z],[n-z]}$. This is intended for the purposes of the current section, although both representations are valid, provided that $z,n-z\in \CC_{\M}$.\\

If a certain nonempty cCoP $\C^o(n,\CC_{\M})$ with an axis $\LL_{[z_1],[z_2]}$ is given, then for another nonempty cCoP $\C^o(n+t,\CC_{\M})$ a lower and an upper limiting axis  $\LL_{[u_1],[u_2]}$ (resp.  $\LL_{[v_1],[v_2]}$) for the validity of the \emph{expansion principles} can be analytically determined  (see Figure \ref{F-Limits}). The area ''squeeze'' marks the scope for an axis $\LL_{[w_1],[w_2]}$ for squeezing. Similarly, the area ''forecast'' marks the scope of $\LL_{[w_1],[w_2]}$ for forecasting. 


\begin{theorem}\label{T-Limits}
Let $\M\subseteq\N$ and $\C^o(n,\CC_{\M})$ and $\C^o(n+t,\CC_{\M})$ be nonempty cCoP with positive integers $n,t\in\mathbb{Q}\subseteq\N$. If there is an axis $\LL_{[z_1],[z_2]}$ of $\C^o(n,\CC_{\M})$, then there are--depending on these axis--a lower (A:) and an upper (B:) limiting axis $\LL_{[u_1],[u_2]}$ (resp.  $\LL_{[v_1],[v_2]}$) of the cCoP $\C^o(n+t,\CC_{\M})$ for their axes $\LL_{[w_1],[w_2]}$ with $\Re(w_1)>\Re(z_1)$ and
\begin{align}
\Re(u_1)&=\Re(z_1) \mbox{ and } \Re(v_2)=\Re(z_2)\label{E-z1u1v2z2}
\end{align}
such that
\begin{align}
A: \Re(u_1)<\Re(w_1)<\Re(v_1)\text{ and }B: \Re(v_2)<\Re(w_2)<\Re(u_2)\label{E-ForSqueeze}
\end{align}
for the Squeeze Principle and
\begin{align}
A: \Re(v_1)<\Re(w_1)<\frac{n+t}{2}\text{ and }B: \frac{n+t}{2}<\Re(w_2)<\Re(v_2)\label{E-ForForecast}
\end{align}
for the Forecast Principle.
\end{theorem}

\begin{proof}
Due to $w_1,z_2\in\CC_B$ and $\LL_{[w_1],[z_2]}\inn \C^o(n+s,\CC_B)$, we deduce $\Re(w_1)+\Re(z_2)=n+s$. Furthermore, we use the requirement of Theorem \ref{T-Expansion}: $\Re(w_1)=\Re(z_1)+s$, so that with \eqref{E-z1u1v2z2} and (\ref{E-ForSqueeze}.A), we get
\begin{align*}
\Re(u_1)&<\Re(w_1)<\Re(v_1)\\
\Re(z_1)&<\Re(w_1)<n+t-\Re(v_2)\\
\Re(z_1)&<\Re(w_1)<n+t-\Re(z_2)\mid +\Re(z_2)\\
\Re(z_1)+\Re(z_2)&<\Re(w_1)+\Re(z_2)<n+t\\
n&<n+s<n+t.
\end{align*} 
We deduce $0<s<t$, which implies the \emph{squeeze principle}.\\

Also, by \eqref{E-z1u1v2z2} and (\ref{E-ForSqueeze}.B), we get
\begin{align*}
\Re(v_2)&<\Re(w_2)<\Re(u_2)\\
\Re(z_2)&<n+t-\Re(w_1)<n+t-\Re(u_1)\mid +\Re(w_1)\\
\Re(z_2)+\Re(w_1)&<n+t<n+t-\Re(z_1)+\Re(w_1)\\
n+s&<n+t<n+t+s\text{, since }\Re(w_1)-\Re(z_1)=s\\
s&<t<t+s.
\end{align*}
We deduce $t>s$, which implies the \emph{squeeze principle}.\\

Again, by \eqref{E-z1u1v2z2} and (\ref{E-ForForecast}.A), we get
\begin{align*}
\Re(v_1)&<\Re(w_1)<\frac{n+t}{2}\\
n+t-\Re(v_2)&<\Re(w_1)<\frac{n+t}{2}\\
n+t-\Re(z_2)&<\Re(w_1)<\frac{n+t}{2}\mid +\Re(z_2) \\
n+t&<\Re(w_1)+\Re(z_2)<\frac{n+t}{2}+\Re(z_2)\\
n+t&<\Re(w_1)+\Re(z_2)<\frac{n+t}{2}+n\\
n+t&<n+s<\frac{n+t}{2}+n.
\end{align*}
We deduce $t<s<\frac{n+t}{2}$, which implies the \emph{forecast principle}.\\

Similarly, by \eqref{E-z1u1v2z2} and (\ref{E-ForForecast}.B), we get
\begin{align*}
\frac{n+t}{2}&<\Re(w_2)<\Re(v_2)\\
\frac{n+t}{2}&<n+t-\Re(w_1)<\Re(v_2)=\Re(z_2)\mid +\Re(w_1)\\
\frac{n+t}{2}+\Re(w_1)&<n+t<\Re(w_1)+\Re(z_2)\\
\frac{n+t}{2}&<n+t<n+s.
\end{align*}
We deduce $t<s$, which implies the \emph{forecast principle}. This completes the proof.
\end{proof}

\begin{figure}
   \includegraphics[width=240pt]{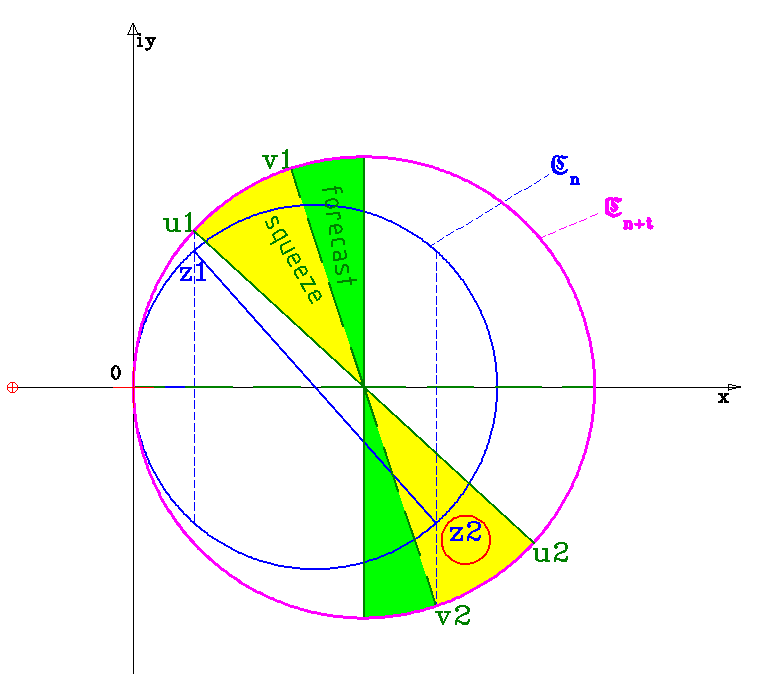}
   \caption{Limiting Axes for the Expansion Principles}\label{F-Limits} 
\end{figure}

\bigskip 

\section*{Conclusion}

The theory developed in this paper shows that the circle-of-partition framework admits a natural and structurally meaningful extension to the complex plane. By imposing the circle condition and working with complex circles of partition, we obtain a geometric model in which additive relations are no longer only algebraically represented but are also encoded through the geometry of embedding circles, axes, chords, and their associated interior and exterior regions. This passage from the real line to the complex plane preserves the additive content of the original construction while revealing a more flexible and visually informative setting in which phenomena of partitions may be analyzed.\\

A central achievement of the paper is the formulation of the expansion principles. The equality, forecast, and squeeze principles provide a systematic way to infer the existence of new admissible generators from previously known complex circles of partition. In particular, they allow one to compare nonempty cCoPs whose generators are related by order or alignment, and to deduce the presence of intermediate or larger generators under appropriate geometric hypotheses. This makes the framework useful not only as a descriptive geometric device, but also as an operative method for producing new partition information from existing data.\\

The analysis of limiting-axis further clarifies the geometry underlying these principles. It shows that the relevant comparison between cCoPs can be expressed in terms of asymptotic positions of their axes, providing a coherent interpretation of how admissible representations evolve as generators vary. In this way, the geometry of the embedding circle serves as more than a convenient illustration: it acts as an organizing principle for the entire theory.\\

The present work suggests several directions for further study. One natural direction is to investigate whether analogous constructions can be developed for higher-dimensional geometric models. Another is to examine how the expansion principles interact with density phenomena and various representation functions. More broadly, the complex setting introduced here indicates that the method of circle-of-partition has the potential to serve as a unifying geometric language for a wider range of additive problems.\\

In summary, the complex circle of partition provides a refined geometric framework for additive partition problems, and the expansion principles offer a robust mechanism for transferring information between admissible generators. The resulting theory strengthens the conceptual reach of circles of partition and opens the way to further structural and asymptotic investigations.


\noindent\rule{100pt}{1pt}

\bibliographystyle{amsplain}

\end{document}